\documentclass[twoside,a4paper,11pt]{amsart}
\usepackage{amsmath,amssymb,amsthm,amscd}
\usepackage[]{latexsym,amssymb,amsmath,amsfonts, amsthm, verbatim}
\usepackage[all,cmtip,ps]{xy}
\usepackage[latin1]{inputenc}
\usepackage[english]{babel}

\usepackage{graphicx}
\newcommand{\DMod}{\mathcal{D}(\text{Mod})}
\newcommand{\dbmod}{\mathcal{D}^b(\text{mod})}

\newcommand{\ma}{\mbox{\rm mod-$A$}}
\newcommand{\mb}{\mbox{\rm mod-$B$}}
\newcommand{\mc}{\mbox{\rm mod-$C$}}
\newcommand{\mai}{\mbox{\rm mod-$A_i$}}
\newcommand{\mbi}{\mbox{\rm mod-$B_i$}}
\newcommand{\mci}{\mbox{\rm mod-$C_i$}}

\newcommand{\dba}{\mathcal{D}^b(\ma)}
\newcommand{\dbb}{\mathcal{D}^b(\mb)}
\newcommand{\dbc}{\mathcal{D}^b(\mc)}
\newcommand{\dbai}{\mathcal{D}^b(\mai)}
\newcommand{\dbbi}{\mathcal{D}^b(\mbi)}
\newcommand{\dbci}{\mathcal{D}^b(\mci)}

\newcommand{\MA}{\mbox{\rm Mod-$A$}}
\newcommand{\MB}{\mbox{\rm Mod-$B$}}
\newcommand{\MC}{\mbox{\rm Mod-$C$}}
\newcommand{\dua}{\mathcal{D}(\MA)}
\newcommand{\dub}{\mathcal{D}(\MB)}
\newcommand{\duc}{\mathcal{D}(\MC)}

\newcommand{\Z}{\mathbb{Z}}

\DeclareMathOperator{\Hom}{Hom}
\DeclareMathOperator{\End}{End}
\DeclareMathOperator{\Ext}{Ext}

\newcommand{\Dcal}{\ensuremath{\mathcal{D}}}

\theoremstyle{plain}
\newtheorem{thm}{Theorem}[section]
\newtheorem{prop}[thm]{Proposition}

\newtheorem{lemma}[thm]{Lemma}
\newtheorem{cor}[thm]{Corollary}

\newtheorem*{ex*}{Example}

\theoremstyle{definition}

\theoremstyle{remark}
\newtheorem*{rem}{Remark}

\newcommand{\ra}{\rightarrow}

\newcommand{\ten}{\otimes}

\begin{document}

\title{Blocks of group algebras are derived simple}
\author{Qunhua Liu}
\address{Qunhua Liu, Institute for algebra and number theory, University of Stuttgart, Pfaffenwaldring 57, D-70569 Stuttgart, Germany}
\email{qliu@mathematik.uni-stuttgart.de}
\author{Dong Yang}
\address{Dong Yang, Hausdorff Research Institute for Mathematics, Poppelsdorfer Allee 82, D-53115 Bonn, Germany}
\email{dongyang2002@gmail.com}
\date{\today}

%= Abstract ========================================================
%\bigskip

\begin{abstract}
A derived version of Maschke's theorem for finite groups is proved:
the derived categories, bounded or unbounded, of all blocks of the
group algebra of a finite group are simple, in the sense that they
admit no nontrivial recollements. This result is independent
of the characteristic of the base field. \smallskip 
\end{abstract}

\maketitle

\section{Introduction}

Let $G$ be any finite group and $k$ a field. An indecomposable
algebra direct summand $B$ of the group algebra $kG$ is called a
{\em block}. By Maschke's theorem, all blocks of $kG$ are simple
algebras if and only if the characteristic of $k$ does not divide
the order of $G$. The main result of this article is a general
statement about blocks that does not need any assumption on the
characteristic of $k$:

\medskip
\noindent{\bf Main Theorem.} {\em All blocks $B$ of $kG$ are derived
simple. More precisely, the bounded derived category $\dbb$ of
finitely generated $B$-modules, as well as the unbounded derived
category $\dub$ of all $B$-modules, are simple in the sense that
they do not admit nontrivial recollements by derived  categories of
the same type.}
\medskip

Simple algebras are, of course, derived simple in this sense. Thus,
when the characteristic does not divide the group order the
statement is an immediate corollary of Maschke's theorem.

The context of the main theorem is the following: Recollements of
triangulated categories, defined by Beilinson, Bernstein and Deligne
\cite{BeilinsonBernsteinDeligne82}, can be seen as analogues of
short exact sequences of these categories. We focus on (bounded or
unbounded) derived categories of finite-dimensional algebras. A
derived category is said to be \emph{simple} if it is nonzero and it is not the
middle term of a nontrivial recollement of derived categories. Once
simpleness has been defined, one can study \emph{stratifications},
i.e. ways of breaking up a given derived category into simple pieces using recollements. They are analogues of composition series for groups/modules.
 Then the question arises which objects
are simple and if a Jordan--H\"older theorem holds true, that is,
whether finite stratifications exist and are unique. Various
positive and negative results recently have been found, see
\cite{AKL2, AKL3, CX1, CX2}. The main theorem provides a large and
quite natural supply of derived simple algebras as well as a derived
Jordan--H\"older theorem for group algebras:

\medskip
\noindent{\bf Corollary.} \emph{Let $G$ be a finite group and $kG$
the group algebra. Then any stratification of
$\mathcal{D}(\mathrm{Mod}$-$kG)$ (respectively,
$\mathcal{D}^b(\mathrm{mod}$-$kG)$) is finite. Moreover, the simple
factors of any two stratifications are the same: they are precisely
the derived categories of the blocks of $kG$.}
\medskip

More generally we will prove derived simpleness for larger classes
of algebras. In the case of the category $\dbmod$ we will show that
all indecomposable symmetric algebras are derived simple. In the
more difficult case of $\DMod$ we will show that indecomposable
symmetric algebras with `enough' cohomology (in a sense made precise
below) are derived simple. Blocks of group algebras as well as indecomposable
symmetric algebras of finite representation type do satisfy this
condition.

The second-named author gratefully acknowledges support from
Max-Planck-Institut f\"ur Mathematik in Bonn and from Hausdorff
Research Institute for Mathematics.
Both authors are deeply grateful to Steffen Koenig for many helpful discussions and suggestions.

%In the algebraic setting, the most interesting recollements are
%those of which all the three terms are (bounded or unbounded)
%derived categories of algebras. Recollements are analogues of short
%exact sequences. A fundamental problem is to understand the simple
%objects: the \emph{derived simple algebras}. Surprising enough, not
%many algebras are known to be derived simple. Only local algebras
%\cite{AKL2}, and some algebras with two simple modules in~\cite{W}
%and~\cite{Happel91}.

%The main aim of this article is to prove the following results:
%\begin{itemize}
%\item[(1)] blocks of group algebras of finite groups over a field are derived
%simple with respect to $D^b(\rfmod)$;
%\item[(2)] blocks of group algebras of finite groups over a field are derived
%simple with respect to $D(\rmod)$.
%\end{itemize}
%We also generalize both results: (1) is generalized to all
%indecomposable symmetric algebras; (2) is generalized to
%indecomposable symmetric algebras satisfying certain conditions.

\section{Recollements and derived simpleness}

Let $k$ be a field. For a finite-dimensional $k$-algebra $A$, we
denote by $\dua$ the derived category of (right) $A$-modules, by
$\dba$ the bounded derived category of finitely generated
$A$-modules, and by $K^b(P_A)$ the homotopy category of bounded
complexes of finitely generated projective $A$-modules. Objects in
$K^b(P_A)$ will be called {\em compact} complexes. We often view
$K^b(P_A)$ as a triangulated full subcategory of the other two
categories, and view $\dba$ as the triangulated full subcategory of
$\dua$ consisting of complexes whose total cohomology space is
finite-dimensional. By abuse of notation we write $\Hom_A(-,-)$ for
both $\Hom_{\mathcal{D}^b(\mathrm{mod}-A)}(-,-)$ and $\Hom_{\mathcal{D}(\mathrm{Mod}{\scriptstyle -}A)}(-,-)$.

\medskip

A \emph{recollement}~\cite{BeilinsonBernsteinDeligne82} of
triangulated categories is a diagram of triangulated categories and
triangle functors
$$\xymatrix@!=3pc{\mathcal{C}' \ar[r]|{i_*=i_!} &\mathcal{C} \ar@<+2.5ex>[l]|{i^!}
\ar@<-2.5ex>[l]|{i^*} \ar[r]|{j^!=j^*} &
\mathcal{C}''\ar@<+2.5ex>[l]|{j_*} \ar@<-2.5ex>[l]|{j_!}}$$ such
that
\begin{enumerate}
\item $(i^\ast,i_\ast)$,\,$(i_!,i^!)$,\,$(j_!,j^!)$ ,\,$(j^\ast,j_\ast)$
are adjoint pairs;

%\smallskip
\item  $i_\ast,\,j_\ast,\,j_!$  are full embeddings;

%\smallskip
\item  $i^!\circ j_\ast=0$ (and thus also $j^!\circ i_!=0$ and
$i^\ast\circ j_!=0$);

%\smallskip
\item  for each $C\in \mathcal{C}$ there are triangles
\begin{eqnarray*}i_! i^!(C)\to C\to j_\ast j^\ast (C)\to i_!i^!(C)[1],\\
j_! j^! (C)\to C\to i_\ast i^\ast(C)\to j_!j^!(C)[1].
\end{eqnarray*}
\end{enumerate}

We are particularly interested in recollements of the following two
forms
$$\xymatrix@!=6.5pc{\dbb \ar[r]|{i_*=i_!} &\dba \ar@<+2.5ex>[l]|{i^!}
\ar@<-2.5ex>[l]|{i^*} \ar[r]|{j^!=j^*} & \dbc\ar@<+2.5ex>[l]|{j_*}
\ar@<-2.5ex>[l]|{j_!} } \ \ \ \ \ \ \ \text{(R1)}$$ and
$$\xymatrix@!=6.5pc{\dub \ar[r]|{i_*=i_!} &\dua \ar@<+2.5ex>[l]|{i^!}
\ar@<-2.5ex>[l]|{i^*} \ar[r]|{j^!=j^*} & \duc\ar@<+2.5ex>[l]|{j_*}
\ar@<-2.5ex>[l]|{j_!} }\ \ \ \ \ \ \ \ \text{(R2)}$$ with $A$, $B$
and $C$ being finite-dimensional algebras over $k$. By
\cite[Corollary 2.3]{AKL3} and \cite[5.2.9]{NZ}, a recollement of
the form (R1) always implies the existence of a recollement of the
form (R2).

\begin{lemma}\label{l:image-from-left}
Let $A$ be a finite-dimensional $k$-algebra admitting a recollement
of the form (R2). Then $j_!j^!(A)$, $i_*i^*(A)$ and $i_*(B)$ all
belong to $\dba$.
\end{lemma}
\begin{proof} There is a canonical triangle
$$j_!j^!(A) \rightarrow A \rightarrow i_*i^*(A)\rightarrow
j_!j^!(A)[1],$$ which yields an isomorphism $\Hom_A(A,i_*i^*(A)[n])\cong
\Hom_A(i_*i^*(A),i_*i^*(A)[n])$ by applying $\Hom_A(-,i_*i^*(A)[n])$. Note that there is an isomorphism between the $n$-th cohomology space  $H^n(i_*i^*(A))$ of $i_*i^*(A)$ and the Hom-space $\Hom_A(A,i_*i^*(A)[n])$ and that $i_*$ is a full embedding. Hence $$H^n(i_*i^*(A))\cong \Hom_A(i_*i^*(A),i_*i^*(A)[n]) \cong \Hom_B(i^*(A),i^*(A)[n]).$$  By \cite[4.3.6, 4.4.8]{NZ}, $i^*(A)$ belongs to and generates $K^b(P_B)$ (i.e. $K^b(P_B)$ is the smallest triangulated subcategory of $\Dcal(\mathrm{Mod}$-$B)$ containing $i^*(B)$ and closed under taking direct summands). Since $B$ is a finite-dimensional algebra, it follows by d\'evissage that the space of self extensions $\Hom_B(i^*(A),i^*(A)[n])$ of $i^*(A)$ is finite-dimensional and vanishes for all but finitely many integers $n$. Combining this observation with the above isomorphism, we obtain that $i_*i^*(A)$ has finite-dimensional total cohomology space. Therefore it belongs to $\dba$. So do $j_!j^!(A)$, thanks to the canonical triangle, and as well as $i_*(B)$, because $i_*(B)$ and $i_*i^*(A)$ generate each other in finitely many steps.
\end{proof}

%\begin{proof} We have a canonical triangle
%$$j_!j^!(A) \rightarrow A \rightarrow i_*i^*(A)\rightarrow
%j_!j^!(A)[1],$$ which yields an isomorphism $i_*i^*(A)\cong
%\rhom_A(i_*i^*(A),i_*i^*(A))$ in $\dua$. by \cite[4.3.6, 4.4.8]{NZ},
%$i_*i^*(A)$ is a compact generator of $i_*\dub$. So
%$\rhom_A(i_*i^*(A),i_*i^*(A))$ is derived equivalent to $B$. Since
%$B$ is finite-dimensional, it follows that $i_*i^*(A)\cong
%\rhom_A(i_*i^*(A),i_*i^*(A))$ has finite-dimensional total
%cohomology. So do $j_!j^!(A)$, thanks to the above triangle, and
%$i_*(B)$, because $i_*(B)$ and $i_*i^*(A)$ generate each other.
%\end{proof}

Derived simpleness of a finite-dimensional $k$-algebra $A$ was
introduced by Wiedemann \cite{W} (see also \cite{AKL2}). By
definition $A$ is said to be {\em derived simple} with respect to
$\dbmod$ respectively $\DMod$, if $A$ is nontrivial and there are no nontrivial recollements
of the form (R1) respectively (R2), namely, none of the full
embedding $i_*$, $j_!$ and $j_!$ is a triangle equivalence. We also
say that $A$ is $\dbmod$\emph{-simple} respectively
$\DMod$\emph{-simple} for short.

A $\Dcal(\mathrm{Mod})$-simple algebra is always
$\Dcal^b(\mathrm{mod})$-simple since, as mentioned above, a recollement
of the form (R1) always induces a recollement of the form (R2). The
converse is in general not true, an example can be found
in~\cite{AKLY01}.

An algebra is said to be {\em indecomposable} if it is not
isomorphic to a direct product of two nonzero algebras. Clearly, if
an algebra is decomposable, then a  nontrivial decomposition of the
algebra yields a nontrivial recollement. Hence a decomposable
algebra is never derived simple in any sense.

\bigskip

\section{The bounded case}

Let $A$ be a finite-dimensional $k$-algebra. Recall that $A$ is said
to be a {\em symmetric algebra}, if $DA$ is isomorphic to $A$ as
$A$-$A$-bimodules, where $D=\Hom_k(-,k)$ is the $k$-dual. In particular, an
$A$-module is projective if and only if it is injective. For
equivalent definitions of symmetric algebras see Curtis and Reiner
\cite{CR}. Group algebras of finite groups form an important class of symmetric algebras.
%Recall that $A$ is {\em Frobenius}, if A is equipped
%with a nondegenerate bilinear form $\sigma: A \times A \ra k$ that
%satisfies the following equation: $\sigma(a\cdot
%b,c)=\sigma(a,b\cdot c)$. Equivalently, one may equip $A$ with a
%linear functional $\lambda:A\ra k$ such that the kernel of $\lambda$
%contains no nonzero left ideal of $A$. A Frobenius algebra is called
%{\em symmetric} if the bilinear form $\sigma$ is symmetric, or
%equivalently $\lambda$ satisfies $\lambda(a\cdot b) = \lambda(b\cdot
%a)$. In particular, $A\cong DA$ as $A$-modules (where
%$D=\Hom_k(-,k)$) and hence $A$ is self-injective. For reference see
%Curtis and Reiner \cite{CR}.

\begin{lemma}[\cite{Rickard02} Corollary 3.2]\label{l:0-cy}
Let $A$ be a symmetric finite-dimensional $k$-algebra. Then there is
a bifunctorial isomorphism
\[D\Hom_A(P,M)\cong\Hom_A(M,P)\] for $P\in K^b(P_A)$, and $M\in \mathcal{D}(\mathrm{mod}$-$A)$.
%In particular, the
%endomorphism algebra $\Hom_A(P,P)$ is symmetric.
\end{lemma}

%\begin{proof}
%By the proof of Theorem I.4.6 in Happel \cite{H1}, the Nakayama
%functor $\nu:= D\Hom_A(-,A) \cong -\otimes_A DA$, sending projective
%$A$-modules to injective $A$-modules, can be naturally extended to a
%triangle equivalence from $K^b(P_A)$ to $K^b(I_A)$. Moreover there
%is a canonical isomorphism
%$$D\Hom_{\dba}(P,-)\cong\Hom_{\dba}(-,\nu P)$$ for $P\in K^b(P_A)$ (viewed as a subcategory of $\dba$).
%Since the algebra $A$ is symmetric, $DA\cong A$ and hence the
%Nakayama functor $\nu$ is just the identity functor. Therefore
%$$D\Hom_{\dba}(P,-)\cong\Hom_{\dba}(-,P)$$ for $P\in K^b(P_A)$.
%\end{proof}

%\medskip

The main result of this section is the following.
\begin{thm} \label{maintheorem}
A finite-dimensional indecomposable symmetric $k$-algebra is $\dbmod$-simple. %In other words, there exists no nontrivial recollements of the form
%$$\xymatrix@!=5pc{\dbb \ar[r] &\dba \ar@<+1.5ex>[l] \ar@<-1.5ex>[l] \ar[r] &\dbc \ar@<+1.5ex>[l] \ar@<-1.5ex>[l]}$$
%with $B$ and $C$ being finite-dimensional $k$-algebras.
\end{thm}

\begin{proof}%[Proof of Theorem~\ref{maintheorem}]
Let $A$ be a finite-dimensional symmetric $k$-algebra and assume
that there exists a nontrivial recollement
of the form (R1). %It follows from the recollement structure that
%$\Img(j_!) = ^\perp i_*(B):=\{X\in\dba: \Hom(X[n],i_*(B))=0,\
%\forall\ n\in\Z\}$ and $\Img(j_*) =
%i_*(B)^\perp:=\{X\in\dba:\Hom(i_*(B),X[n])=0,\ \forall\ n\in \Z\}$.
By \cite[Corollary 2.3]{AKL3}, $j_!j^!(A)$ and $i_*i^*(A)$  belong
to $K^b(P_A)$.
%$i_*(B)$ is compact, i.e. it belongs to $K^b(P_A)$.
Now apply Lemma~\ref{l:0-cy} to $P=j_!j^!(A)$ and $M=i_*i^*(A)[n]$
($n\in\Z$). We obtain
$$D\Hom_A(j_!j^!(A),i_*i^*(A)[n])\cong \Hom_A(i_*i^*(A)[n],j_!j^!(A))$$ for all integers $n$. The left hand side always vanishes as $j^*i_*=0$, and hence so does the right hand side. This means $i_*i^*(A)$ and $j_!j^!(A)$  are orthogonal to each other and the canonical triangle
$$j_!j^!(A) \ra A \ra i_*i^*(A) \ra j_!j^!(A)[1]$$ splits, i.e.
$A\cong j_!j^!(A) \oplus i_*i^*(A)$.
%In particular, $D\Hom_{\dba}(i_*(B),-)\cong\Hom_{\dba}(-,i_*(B))$. This
%implies the left and right perpendicular categories of $i_*(B)$
%coincide, i.e. $\Img(j_!) = \Img(j_*)$. But then $\Img(j_!)^\perp =
%\Img(i_*)= ^\perp \Img(j_*) = ^\perp \Img(j_!)$, which means
%$\Img(j_!)$ and $\Img(i_*)$ are orthogonal to each other and each
%canonical triangle $j_!j^!(X) \ra X \ra i_*i^*(X) \ra j_!j^!(X)[1]$
%($X\in\dba$) splits.
In particular, $A=\End_A(A)$ is decomposed as the product of
$\End_A(j_!j^!(A))$ and $\End_A(i_*i^*(A))$.
\end{proof}

\begin{rem} Lemma \ref{l:0-cy} says that the triangulated category $K^b(P_A)$ is $0$-Calabi--Yau. The above proof shows that this $0$-Calabi--Yau triangulated category is simple, in the sense that it admits no nontrivial recollements of triangulated categories. Notice that in the proof the Calabi--Yau dimension is not important. Thus with the same proof one shows that an indecomposable $d$-Calabi--Yau triangulated category is simple ($d\in\Z$)\footnote{We thank Bernhard Keller for pointing out this to us.}. In fact, it admits no nontrivial stable t-structures. Cluster categories of connected quivers are examples of indecomposable $2$-Calabi-Yau categories, and hence are simple.
\end{rem}

As a corollary, we establish a derived Jordan--H\"older theorem for
symmetric algebras, which is on the existence and uniqueness of
finite stratifications. Roughly speaking, a \emph{stratification} is
a way of breaking up a given derived category into simple pieces
using recollements. More rigorously, a stratification is a full
rooted binary tree whose root is the given derived category, whose
nodes are derived categories and whose leaves are simple (they are
called the \emph{simple factors} of the stratification) such that a
node is a recollement of its two child nodes unless it is a leaf.
For a finite-dimensional algebra, a \emph{block} is an
indecomposable algebra direct summand. The number of blocks is a
derived invariant.

\begin{cor}\label{c:jh-db} Let $A$ be a finite-dimensional symmetric algebra. Then any stratification of $\mathcal{D}^b(\mathrm{mod}$-$A)$ is finite. Moreover, the simple
factors of any two stratifications are the same: they are precisely the bounded derived categories of the blocks of $A$.
\end{cor}

\begin{proof} Suppose the algebra $A$ has $s$ blocks $A_i$ with $A=\bigoplus_{i=1}^s A_i$. Suppose a recollement of $\Dcal^b(\mathrm{mod}$-$A)$ of the form (R1) is given. The block decomposition of $A$ yields a decomposition of its derived category: $\Dcal^b(\mathrm{mod}$-$A)=\bigoplus_i \Dcal^b(\mathrm{mod}$-$A_i)$. In particular, $j_!(C)$ is a direct sum $\bigoplus_i X_i$ of $X_i\in\Dcal^b(\mathrm{mod}$-$A_i)$, and hence $C$, being isomorphic to $\End_A(j_!(C))$, admits a block decomposition $C=\bigoplus_i C_i$ such that $j_!(C_i)=X_i$. Similarly, the algebra $B$ admits a block decomposition $B=\bigoplus_i B_i$ such that $i_*(B_i)\in\mathcal{D}^b(\mathrm{mod}$-$A_i)$.  Fix an $i=1,\ldots,s$. For an indecomposable object $M\in\Dcal^b(\mathrm{mod}$-$C_i)$, considered as an object in $\Dcal^b(\mathrm{mod}$-$C)$, there exists an $n\in\mathbb{Z}$ such that $\Hom_{C}(C_i,M[n])\neq 0$. Since $j_!$ is fully faithful, it follows that $\Hom_A(X_i,j_!(M)[n])\neq 0$, implying that $j_!(M)\in\Dcal^b(\mathrm{mod}$-$A_i)$. Therefore, $j_!$ restricts to a triangle functor $j_!:\Dcal^b(\mathrm{mod}$-$C_i)\rightarrow \Dcal^b(\mathrm{mod}$-$A_i)$. Similarly, $i_*$ restricts to a triangle functor $i_*:\Dcal^b(\mathrm{mod}$-$B_i)\rightarrow \Dcal^b(\mathrm{mod}$-$A_i)$. For an object $N\in\Dcal^b(\mathrm{mod}$-$A_i)$, we have $\Hom_{C}(C_j,j^*(N))=\Hom_A(j_!(C_j),N)=\Hom_A(X_j,N)=0$ for $j\neq i$, implying that $j^*(N)\in\Dcal^b(\mathrm{mod}$-$C_i)$. Therefore, $j^*$ restricts to a triangle functor $j^*:\Dcal^b(\mathrm{mod}$-$A_i)\rightarrow \Dcal^b(\mathrm{mod}$-$C_i)$. Similarly, $j_*$, $i^*$ and $i^!$ respectively restricts to triangle functors $j_*:\Dcal^b(\mathrm{mod}$-$C_i)\rightarrow \Dcal^b(\mathrm{mod}$-$A_i)$, $i^*:\Dcal^b(\mathrm{mod}$-$A_i)\rightarrow \Dcal^b(\mathrm{mod}$-$B_i)$ and $i^!:\Dcal^b(\mathrm{mod}$-$A_i)\rightarrow \Dcal^b(\mathrm{mod}$-$B_i)$. It follows that the diagram below is a recollement
 $$\xymatrix@!=6.5pc{\dbbi \ar[r]|{i_*=i_!} &\dbai \ar@<+2.5ex>[l]|{i^!}
\ar@<-2.5ex>[l]|{i^*} \ar[r]|{j^!=j^*} & \dbci\ar@<+2.5ex>[l]|{j_*}
\ar@<-2.5ex>[l]|{j_!} .}$$
By Theorem \ref{maintheorem}, the blocks $A_i$ are derived simple. Therefore for each $i$, either $B_i=A_i$ and $C_i=0$ or $B_i=0$ and $C_i=A_i$, up to derived equivalence. In particular, up to derived equivalence, $B$ and $C$ are algebra direct summands of $A$. Therefore, the desired result follows by induction on the number $s$ of blocks of $A$.
\end{proof}

\bigskip

%%%%%%%%%%%%%%%%%%%%%%%%%%%%%%%%%%%%%%%%%%%%%%%%%%%%%%%%%%%%%%%%%%%%%%%%%%%%%%%%%%%%%%%%%%
\section{The unbounded case}

In this section we partially generalise Theorem~\ref{maintheorem} to
the unbounded derived category.
Let $A$ be a finite-dimensional $k$-algebra. Consider the following
condition
\begin{itemize}
\item[(\#)] for any finitely generated non-projective $A$-module $M$, there are infinitely many integers
$n$ with $\Ext_A^n(M,M) \neq 0$.
\end{itemize}
We will prove that an indecomposable finite-dimensional symmetric algebra satisfying (\#) is $\mathcal{D}(\mathrm{Mod})$-simple.

Here is an example of a symmetric algebra which satisfies (\#). Let $A$ be the quotient of the path algebra of the quiver
$$\xymatrix{1\ar@<.7ex>[r]^{\alpha}&2\ar@<.7ex>[l]^{\beta}}$$ by the
ideal generated by $\alpha\beta\alpha$ and $\beta\alpha\beta$. Up to
isomorphism $A$ has four indecomposable non-projective modules given
by the following representations
\[\xymatrix{k\ar@<-.7ex>[r]&  0\ar@<-.7ex>[l]},~~ \xymatrix{0\ar@<-.7ex>[r]&k\ar@<-.7ex>[l]},
~~\xymatrix{k\ar@<-.7ex>[r]_{1}&k\ar@<-.7ex>[l]_{0}},
~~\xymatrix{k\ar@<-.7ex>[r]_{0}&k\ar@<-.7ex>[l]_{1}}.\] It is easy
to check that
\begin{eqnarray*}
\Ext_A^n(M,M)&=&\begin{cases}k & \text{if } n\geq 0 \text{ and }
n\equiv 0,3\hspace{-7pt}\pmod{4},\\ 0 & \text{otherwise},\end{cases}
%\Ext_A^n(M_i,M_i)&=&\begin{cases}k & \text{if } n\geq 0 \text{ and }
%n\equiv 0,3\hspace{-7pt}\pmod{4},\\ 0 &
%\text{otherwise}.\end{cases}
\end{eqnarray*}
for each of the above modules $M$.
%Thus the condition (\#) is satisfied. By Theorem~\ref{t:derived-sim-Mod}, the algebra $A$ is derived simple
%with respect to $D(\rmod)$.

\medskip
More generally, the condition (\#) is satisfied by the following two classes of algebras, by~\cite[Proposition
3.2, Example 3.1]{S}
\begin{itemize}
\item[(1)] group algebras of finite groups over $k$;
\item[(2)] self-injective $k$-algebras of finite representation type.
\end{itemize}

\medskip

Recall that a finite-dimensional algebra is said to be
\emph{self-injective} if all projective modules are injective.

\begin{prop}\label{p:image-from-left}
Let $A$ be a finite-dimensional self-injective $k$-algebra
satisfying the condition (\#). Assume that $A$ admits a nontrivial
recollement of the form (R2). Then $i_*i^*(A)$ belongs to
$K^b(P_A)$.
\end{prop}

\begin{proof} Let $A$ be a finite-dimensional $k$-algebra satisfying the condition
(\#). We claim that for any $X\in\dba$ we have either $X\in
K^b(P_A)$ or there are infinitely many integers $n$ such that
$\Hom_A(X,X[n])\neq 0$. By Lemma~\ref{l:image-from-left}, the object $i_*i^*(A)$ belongs to $\dba$, and hence $\Hom_A(i_*i^*(A),i_*i^*(A)[n])\cong H^n(i_*i^*(A))$ vanishes for all but finitely many $n\in\mathbb{Z}$. Therefore $i_*i^*(A)$ must be in $K^b(P_A)$.

\medskip

Next we prove the claim. Let $X\in \dba$. Without loss of
generality, we assume that $X$ is a minimal complex of finitely
generated projective $A$-modules which is bounded from the right.
Here minimality means that for any $p\in\mathbb{Z}$ the image
of each differential $d^p:X^p\rightarrow X^{p+1}$ lies in the
radical of $X^{p+1}$. Since $X$ has bounded cohomology, there is an
integer $N$ such $H^n(X)=0$ for $n\leq N$. Up to shift, we may
assume that $N=0$. Let $X''$ be the subcomplex of $X$ with
$(X'')^n=X^n$ for $n> 0$ and $(X'')^n=0$ for $n\leq 0$ and let $X'$
be the corresponding quotient complex. Then there is a triangle
\[X''\rightarrow X\rightarrow X'\rightarrow X''[1].\]
Note that $X''\in K^b(P_A)$, and $X'$ has cohomology concentrated in
degree $0$, and hence it is the minimal projective resolution of a
finitely generated $A$-module, say $M$.

Case 1: $M$ is projective. This implies that $X'$ is a stalk
complex, and hence $X\in K^b(P_A)$.

Case 2: $M$ is not projective. The above triangle gives us two long
exact sequences
\[\ldots\rightarrow \Hom_A(X,X''[n])\rightarrow \Hom_A(X,X[n])\rightarrow \Hom_A(X,X'[n])
\rightarrow\ldots\]
\[\ldots\rightarrow \Hom_A(X',X'[n])\rightarrow
\Hom_A(X,X'[n])\rightarrow \Hom_A(X'',X'[n]) \rightarrow\ldots\]
Recall that $X''\in K^b(P_A)$ is also a bounded complex of finitely generated injective modules. Therefore there exist only finitely
many integers $n$ such that $\Hom_A(X,X''[n])\neq 0$ (respectively,
$\Hom_A(X'',X'[n])\neq 0$). So from the above two long exact
sequences we see that
\[\Hom_A(X,X[n])\cong \Hom_A(X,X'[n])\cong \Hom_A(X',X'[n])\]
for all but finitely many integers $n$.
Now the claim follows from the condition (\#) since
$\Hom_A(X',X'[n])=\Ext_A^n(M,M)$.
\end{proof}

Now we are ready to prove the main results of this section.

\begin{thm}\label{t:derived-sim-Mod} Let $A$ be a finite-dimensional symmetric $k$-algebra satisfying the condition (\#). If $A$ is indecomposable, then it
is $\Dcal(\mathrm{Mod})$-simple.
\end{thm}

\begin{proof} Let $A$ be finite-dimensional symmetric satisfying the condition (\#). It follows from Proposition~\ref{p:image-from-left} that $i_*i^*(A)$
belongs to $K^b(P_A)$. Now we proceed as in the proof of
Theorem~\ref{maintheorem} to show that $A$ is decomposable.
\end{proof}

%We have the following corollary.% of Schulz \cite[3.1,3.2]{S} and the
%preceding theorem.

\begin{cor} The following two classes of finite-dimensional symmetric algebras are $\Dcal(\mathrm{Mod})$-simple:
\begin{enumerate}
\item blocks of group algebras of finite groups over $k$;
\item indecomposable symmetric $k$-algebras of finite representation type.
\end{enumerate}
\end{cor}

\begin{proof} This is an immediate consequence of Theorem~\ref{t:derived-sim-Mod}. For completeness, we give a proof for the assertion that group algebras of finite groups satisfy the condition
(\#). Let $G$ be a finite group and let $A=kG$ be the group algebra.
For a finitely generated $A$-module $M$, define
$$\Ext^{\cdot}(M,M)=\begin{cases} \bigoplus_{n\geq
0}\Ext_A^{n}(M,M) & \text{if } \mathrm{char}(k)=2,\\
\bigoplus_{n\geq 0}\Ext_A^{2n}(M,M) & \text{otherwise.}\end{cases}$$
Let $k$ also denote the trivial module. Then the graded $k$-algebra
$H^\cdot(G,k)=\Ext^\cdot(k,k)$ is commutative Noetherian.  Let $M$
be a finitely generated $A$-module. The tensor product $-\ten_k M$
induces an algebra homomorphism
\[\varphi:H^{\cdot}(G,k)\rightarrow \Ext^\cdot(M,M).\]
It is known that if $M$ is non-projective, then the Krull dimension
of $H^{\cdot}(G,k)/\ker(\varphi)$ is greater than or equal to $1$
(see for example~\cite[Section 2.24, 2.25]{Benson84}), which implies
that it is infinite-dimensional. As a consequence, $\Ext^\cdot(M,M)$
is infinite-dimensional, implying the condition (\#).
\end{proof}

We have a $\Dcal(\mathrm{Mod})$-counterpart of Corollary~\ref{c:jh-db}: an unbounded derived Jordan--H\"older theorem for symmetric algebras satisfying the condition (\#).

\begin{cor}\label{c:jh-du} Let $A$ be a finite-dimensional symmetric algebra satisfying the condition (\#). Then any stratification of $\mathcal{D}(\mathrm{Mod}$-$A)$ is finite. Moreover, the simple
factors of any two stratifications are the same: they are precisely the derived categories of the blocks of $A$.
\end{cor}
\begin{proof} Similar to the proof for Corollary~\ref{c:jh-db}.
\end{proof}

\bigskip

%%%%%%%%%%%%%%%%%%%%%%%%%%%%%%%%%%%%%%%%%%%%%%%%%%%%%%%%%%%%%%%%%%%%%%%%%%%%%%%%%%%%%%%%%%

\end{document}